\documentclass[a4paper, 10pt, conference]{ieeeconf}      

\IEEEoverridecommandlockouts                              
\usepackage{hyperref}
\usepackage{amsfonts}
\usepackage{amsmath}
\usepackage{amssymb}
\usepackage{makeidx}
\usepackage{epsfig}
\usepackage{graphicx}

\def\R{\mathcal{R}}
\def\P{\mathcal{P}}
\def\O{\mathcal{O}}

\def\DR{{\textbf{D}}_\R}
\def\DO{{\textbf{D}}_\O}
\def\DDR{{D}_\R}
\def\DDO{{D}_\O}

\newtheorem{theorem}{Theorem}

\newtheorem{definition}{Definition}
\newtheorem{example}{Example}
\newtheorem{lemma}[theorem]{Lemma}
\newtheorem{proposition}[theorem]{Proposition}
\newtheorem{problem}{Problem}

\newtheorem{remark}{Remark}

\title{\LARGE \bf Optimal co-design of control, scheduling and routing\\
in multi-hop control networks}
\author{F. Smarra, A. D'Innocenzo and M.D. Di Benedetto
\thanks{The authors are with the Department of Electrical and Information Engineering, Center of Excellence DEWS, University of L'Aquila, Italy. The research leading to these results has received funding from the European Union Seventh Framework Programme [FP7/2007-2013] under grant agreement n°257462 HYCON2 Network of excellence}}

\begin{document}

\maketitle
\thispagestyle{empty}
\pagestyle{empty}


\begin{abstract}
A Multi-hop Control Network consists of a plant where the communication between sensors, actuators and computational units is supported by a (wireless) multi-hop communication network, and data flow is performed using scheduling and routing of sensing and actuation data. Given a SISO LTI plant, we will address the problem of co-designing a digital controller and the network parameters (scheduling and routing) in order to guarantee stability and maximize a performance metric on the transient response to a step input, with constraints on the control effort, on the output overshoot and on the bandwidth of the communication channel. We show that the above optimization problem is a polynomial optimization problem, which is generally NP-hard. We provide sufficient conditions on the network topology, scheduling and routing such that it is computationally feasible, namely such that it reduces to a convex optimization problem.
\end{abstract}


\section{Introduction} \label{secIntro}

Wireless networked control systems are spatially distributed control systems where the communication between sensors, actuators, and computational units is supported by a shared wireless communication network. Control with wireless technologies typically involves multiple communication hops for conveying information from sensors to the controller and from the controller to actuators. The main motivation for studying such systems is the emerging use of wireless technologies in control systems (see e.g. \cite{akyildiz_wireless_2004}, \cite{SongIECON2010} and references therein). The use of wireless Multi-hop Control Networks (MCNs) in industrial automation results in flexible architectures and generally reduces installation, debugging, diagnostic and maintenance costs with respect to wired networks. Although MCNs offer many advantages, their use for control is a challenge when one has to take into account the joint dynamics of the plant and of the communication protocol. Wide deployment of wireless industrial automation requires substantial progress in wireless transmission, networking and control, in order to provide formal models and verification/design methodologies for MCNs. In particular, co-design of the control algorithm and of the network configuration for a MCN requires addressing issues such as scheduling and routing using real communication protocols.

Recently, a huge effort has been made in scientific research on Networked Control Systems (NCSs), see \cite{Zhang2001},~\cite{WalshCSM2001},~\cite{SpecialIssueNCS2004},~\cite{Arzen06},~\cite{Hespanha2007},~\cite{MurrayTAC2009},~\cite{HeemelsTAC10}, \cite{HeemelsTAC11} and references therein for a general overview. In the current research on the interaction between control networks and communication protocols, most efforts in the literature focus on scheduling message and sampling time assignment for sensors/actuators and controllers interconnected by wired common-bus networks, e.g. \cite{Astrom97j1},~\cite{walsh_stability_2002},~\cite{AlHammouri2006},~\cite{TabbaraTAC2007}, and \cite{TabbaraTAC2008}.

In general, the literature on NCSs addresses non--idealities (such as quantization errors, packets dropouts, variable sampling and delay and communication constraints) as aggregated network performance variables, losing irreversibly the dynamics introduced by scheduling and routing communication protocols. What is needed for modeling and analyzing control protocols on multi hop control networks is an integrated framework for analysing/co-designing network topology, scheduling, routing and control. In~\cite{Andersson:CDC05}, a simulative environment of computer nodes and communication networks interacting with the continuous-time dynamics of the real world is presented. To the best of our knowledge, the only formal model of a MCN has been presented in \cite{AlurTAC11}, where the modeling and stability verification problem has been addressed for a MIMO LTI plant embedded in a MCN, with the controller already designed. A mathematical framework has been proposed, that allows modeling the MAC layer (communication scheduling) and the Network layer (routing) of the recently developed wireless industrial control protocols, such as WirelessHART (\texttt{www.hartcomm2.org}) and ISA-100 (\texttt{www.isa.org}). The mathematical framework defined in \cite{AlurTAC11} is compositional, namely it is possible to exploit compositional operators of automata to design scalable scheduling and routing for multiple control loops closed on the same multi-hop communication network. In \cite{DiBenedettoIFAC11Stab}, we extended the MCN definition to model redundancy in data communication (i.e. sending sensing and actuation data through multiple paths), with the aim of rendering the system fault tolerant with respect to link failures and robust with respect to packet losses. In \cite{PappasCDC2011} a MCN is defined as an autonomous system where the wireless network \emph{itself} acts as a fully decentralized controller, and differs from the model we use in \cite{DiBenedettoIFAC11Stab} and in this paper, where the wireless network \emph{transfers} sensing and actuation data between a plant and a centralized controller. Moreover, in our model we explicitly take into account the effect of the scheduling ordering of the node transmissions in the sensing and actuation data relay, which provides more information when modeling the transient response of the system.

\begin{figure}[ht]
\begin{center}
\vspace{-0.4cm}
\includegraphics[width=0.5\textwidth]{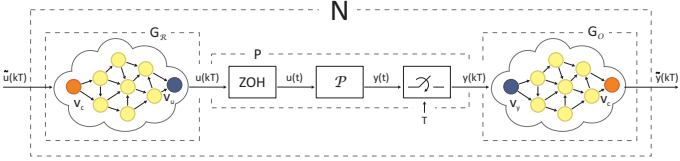}
\vspace{-0.5cm}
\caption{Multi-hop Control Network scheme.} \label{MCNscheme}
\vspace{-0.5cm}
\end{center}
\end{figure}

In this paper we consider a MCN that consists of a controllable and observable continuous-time LTI SISO plant interconnected to a digital controller via two (wireless) multi-hop communication networks, as illustrated in Figure \ref{MCNscheme}. An example of such networked control scheme is the mine application investigated in \cite{D'InnocenzoCASE2009}, where an industrial plant is connected to a controller via a wireless multi-hop communication network. Since we assume that the data transfer is performed using multiple paths and merging redundant data by linear combination, the communication networks will be characterized by LTI SISO dynamics that depend on the network configuration (topology, scheduling and routing). In \cite{DiBenedettoIFAC11Stab} we showed that, no matter how complex is the topology of the network, it is always possible to design a controller and a network configuration (scheduling and routing) that guarantees asymptotic stability of the networked closed loop system. In this paper we address the problem of co-designing a digital controller and the network parameters (scheduling and routing) in order to guarantee stability and maximize a performance metric on the transient response to a step input. More precisely, we will minimize the $\mathcal L_2$ norm of the error signal while requiring the closed loop system to follow a step reference with zero steady state error in finite time (deadbeat control), with constraints on the maximum overshoot in the input and output signal of the plant and on the maximum data rate on the communication links.

We first assume that topology and scheduling are given and show that the above optimization is a polynomial optimization problem, which is generally NP-hard, and we provide sufficient conditions on the network parameters such that it reduces to a convex optimization problem. Then, we consider the scheduling as a design variable, and we exploit the above optimization problem to define a total ordering induced by the $\mathcal L_2$ norm of the error signal on the finite set of scheduling functions. This allows to compute the optimal scheduling with a combinatorial computational complexity with respect to the number of communication links of the network. In future extensions of this paper we aim at providing algorithms to reduce such computational complexity.

The paper is organized as follows. In Section \ref{secMCNModeling} we provide the model of a MCN, while in Section \ref{secMCNDataRateDef} we define the transmission data rate of each node. In Section \ref{secCoDesignProblem} we define the co-design optimization problem and provide the main results of the paper.


\section{Modeling of MCN} \label{secMCNModeling}

To allow systematic methods for designing the communication protocol configuration in a MCN, a mathematical model of the effect of scheduling and routing on the control system is needed. Definition \ref{defMCN} allows modeling time-triggered communication protocols that specify TDMA, FDMA and/or CDMA access to a shared communication resource, for a set of communication nodes interconnected by an arbitrary acyclic radio connectivity graph. In particular, it allows modeling wireless multi-hop communication networks that implement protocols such as WirelessHART and ISA-100.
\begin{figure}[ht]
\begin{center}
\vspace{-0.3cm}
\includegraphics[width=0.3\textwidth]{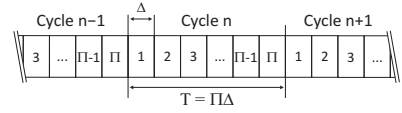}
\vspace{-0.3cm}
\caption{Time-slotted structure of frames.}\label{frame}
\vspace{-0.6cm}
\end{center}
\end{figure}

In these standards the access to the shared communication channel is specified as follows: time is divided into slots of fixed duration $\Delta$, and groups of $\Pi$ time slots are called frames of duration $T = \Pi \Delta$ (see Figure \ref{frame}). For each frame, a communication scheduling allows each node to transmit data only in a specified time slot. The scheduling is periodic with period $\Pi$, i.e. it is repeated in all frames.

\begin{definition}\label{defMCN}
A SISO Multi-hop Control Network is a tuple $\mathcal{N} = ( \P, G_\R, \eta_{\R}, G_\O, \eta_{\O}, \Delta)$ where:
\begin{itemize}
\item $\P(s)$ is a transfer function that models the dynamics of a continuous-time SISO LTI system.

\item $G_\R = ( V_{\R},E_{\R}, W_{\R} )$ is the controllability radio connectivity acyclic graph, where the vertices correspond to the nodes of the network, and an edge from $v$ to $v'$ means that $v'$ can receive messages transmitted by $v$ through the wireless communication link $(v,v')$. We denote $v_{c}$ the special node of $V_{\R}$ that corresponds to the controller, and $v_u \in V_{\R}$ the special node that corresponds to the actuator of the input $u$ of $\P$. The weight function $W_{\R} : E_{\R} \to \mathbb R$ associates to each link a real constant. The role of $W_{\R}$ will be clear in the following definition of $\eta_{\R}$.

\item $\eta_{\R} \colon \mathbb N \to 2^{E_{\R}}$ is the controllability scheduling function, that associates to each time slot of each frame a set of edges of the controllability radio connectivity graph. Since in this paper we only consider a periodic scheduling that is repeated in all frames, we define the controllability scheduling function by $\eta_{\R} \colon \{1, \ldots, \Pi\} \to 2^{E_{\R}}$. The integer constant $\Pi$ is the period of the controllability scheduling. The semantics of $\eta_{\R}$ is that $( v,v' ) \in \eta(h)$ if and only if at time slot $h$ of each frame the data content of the node $v$ is transmitted to the node $v'$. We assume that each link can be scheduled only one time for each frame. This does not lead to loss of generality, since it is always possible to obtain an equivalent model that satisfies this constraint by appropriately splitting the nodes of the graph, as already illustrated in the memory slot graph definition of \cite{AlurTAC11}.


\item $G_\O = ( V_{\O},E_{\O}, W_{\O} )$ is the observability radio connectivity acyclic graph, and is defined similarly to $G_\R$. We denote with $v_{c}$ the special node of $V_{\O}$ that corresponds to the controller, and $v_y \in V_{\O}$ the special node that corresponds to the sensor of the output $y$ of $\P$.

\item $\eta_{\O} \colon \{1, \ldots, \Pi\} \to 2^{E_{\O}}$ is the observability scheduling function, and is defined similarly to $\eta_{\R}$. We remark that $\Pi$ is the same period as the controllability scheduling period.


\item $\Delta$ is the time slot duration. As a consequence, ${T} = \Pi \Delta$ is the frame duration.
\end{itemize}
\end{definition}

\begin{figure}[ht]
\begin{center}
\vspace{-0.5cm}
\includegraphics[width=0.4\textwidth]{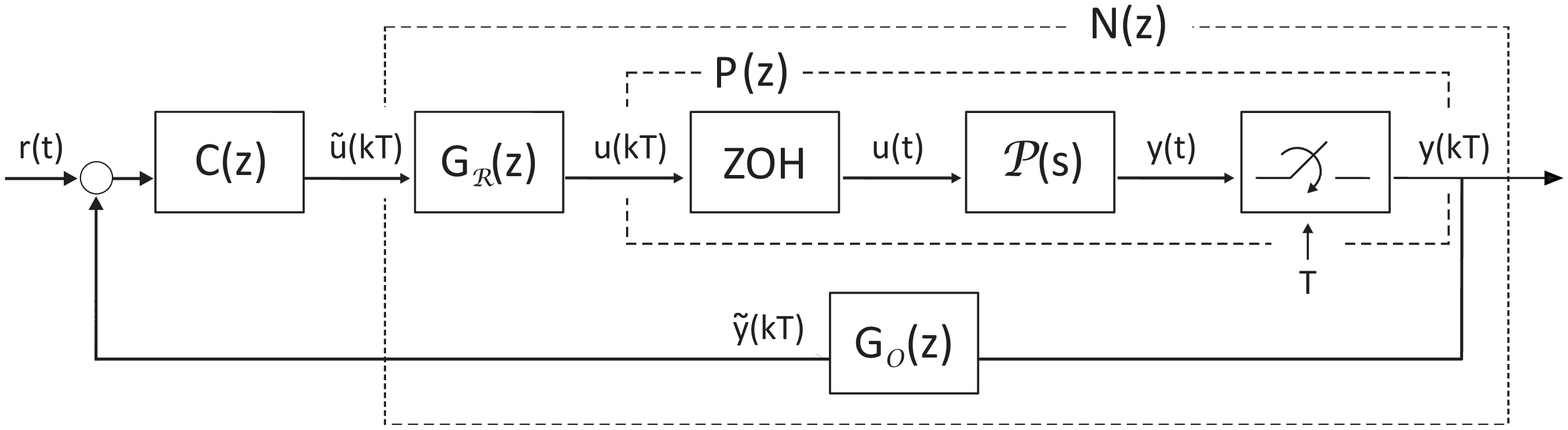}
\vspace{-0.3cm}
\caption{MCN closed loop control scheme.} \label{MCNblocks}
\vspace{-0.5cm}
\end{center}
\end{figure}


For any given radio connectivity graph that models the communication range of each node, designing a scheduling function induces a communication scheduling (namely the time slot when each node is allowed to transmit) and a routing (namely the set of paths that convey data from the input to the output of the connectivity graph) of the communication protocol. Since the scheduling function is periodic the induced communication scheduling is periodic, and the induced routing is static. The above definition allows modeling redundancy in data communication, i.e. sending the same control data through multiple paths in the same frame. This kind of redundancy is called \emph{multi-path routing} (or \emph{flooding}, in the \emph{communication} scientific community), and aims at rendering the MCN robust with respect to node failures and packet losses. To define the semantic of the data-flow in the networks we assume that each communication node computes a linear combination of the data received from the incoming links according to a weight function (i.e. $W_{\R}$ or $W_{\O}$), and forwards these data to the next links. Note that this forwarding protocol does not require the use of a reordering protocol (such as TCP) and of a timestamp in the packets to indicate how recent is the data. This reduces the complexity of the protocol and the bit length of each packet (and thus the data traffic in the network).

As derived in \cite{DiBenedettoIFAC11Stab} the dynamics of a MCN $N$ can be modeled as in Figure \ref{MCNblocks}, where each block is a discrete-time SISO LTI system with sampling time equal to the frame duration, characterized by the transfer functions $G_{\R}(z)$, $P(z)$ and $G_{\O}(z)$. Throughout the paper, given a transfer function $F(z)$, we will denote by $N_F$ and $D_F$ respectively its numerator and denominator. We derive the discrete-time transfer function representation of the plant input/output dynamics $P(z)$ by discretizing $\P(s)$ with sampling time $T$:
$$
P(z)=\frac1{M(z)}P'(z)=\frac{b_{n-1}z^{n-1}+b_{n-2}z^{n-2}+\ldots+b_0}{M(z)(z^r+a_{r-1}z^{r-1}+\ldots+a_0)},
$$
where $M(z)$ is the polynomial whose roots are the stable poles of $P(z)$. Since $P(z)$ is derived by discretizing a continuous-time transfer function, the degree $r+deg(M(z))$ of the denominator is equal to $n$. As illustrated in \cite{DiBenedettoIFAC11Stab}, the transfer function of the controllability graph $G_{\R}(z)$ can be expressed as follows:
\begin{equation}\label{eqGammaDef}
G_{\R}(z) =  \frac{\sum\limits_{d \in \DR}\gamma_{\R}(d)z^{\DDR-d}}{z^{\DDR}}, \qquad \gamma_{\R}(d) = \sum\limits_{\rho \in \chi_{\R}(d)} W_{\R}(\rho),
\end{equation}
where $\chi_{\R}(d)$ is the set of paths of $G_\R$ characterized by delay $d$, $\DR \subset \mathbb N$ is the set of the $|\DR|$ different delays introduced by the paths of $G_{\R}$, $\DDR = \max\{\DR\}$ is the maximum introduced delay, $W_{\R}(\rho)$ is the product of the weights of all links that generate path $\rho$ of $G_\R$. We define the vector $\boldsymbol\gamma_{\R} = [\gamma_{\R}(d), d \in \DR] \in \mathbb R^{|\DR|}$.
The transfer function of the observability graph $G_{\O}(z)$ is derived similarly.

According to the scheduling function, different paths can be characterized by different delays. In \cite{DiBenedettoIFAC11Stab} we illustrated that, in order to avoid issues related to zero-pole cancelations between the controllability and observability networks and the plant (which leads to loss of controllability and observability), it is convenient to design a scheduling function such that all paths are characterized by the same delay, i.e. $|\DR|=|\DO|=1$. Moreover, as illustrated in Section \ref{secCoDesignProblem}, this situation is also convenient since it reduces the computational complexity of the optimization problem addressed in this paper. However, the following example shows that designing such a scheduling can increase the frame duration and therefore the total time delay introduced by the network, and it is not necessarily the best solution in terms of system performance, as illustrated in Example \ref{exFinalExample}. When multiple control loops are closed on the same MCN it is even more difficult to design such a scheduling, as illustrated in \cite{D'InnocenzoCASE2009}.

\begin{figure}[ht]
\begin{center}
\vspace{-0.4cm}
\includegraphics[width=0.5\textwidth]{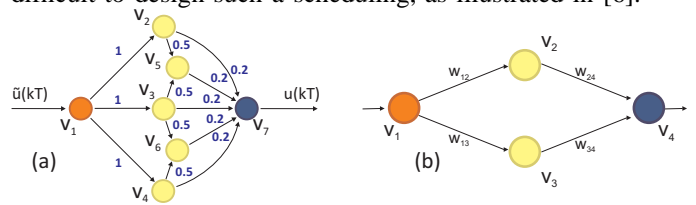}
\vspace{-0.3cm}
\caption{Controllability graphs of Example \ref{DelayExample} and \ref{exOvershoot} (a), and Example \ref{exDataRate} (b).}\label{figExamples}
\vspace{-0.5cm}
\end{center}
\end{figure}

\begin{example}\label{DelayExample}
Let us consider the controllability graph illustrated in Figure \ref{figExamples}(a) and assume that, because of interference in the shared wireless channel, simultaneous transmission is only allowed for the set of nodes $\{v_1, v_5, v_6\}$ and $\{v_2, v_3, v_4\}$. Let us suppose $\Delta=10\ ms$, and consider the following controllability scheduling functions $\eta_\R^a$ and $\eta_\R^b$:
\scriptsize
\begin{align*}
\eta_\R^a(1) = \{&(v_1,v_2),(v_1,v_3),(v_1,v_4)\},\\
\eta_\R^a(2) = \{&(v_2,v_5),(v_2,v_7),(v_3,v_5),(v_3,v_6),(v_3,v_7),(v_4,v_6),(v_4,v_7)\},\\
\eta_\R^a(3) = \{&(v_5,v_7),(v_6,v_7)\}.\\
\eta_\R^b(1) = \{&(v_1,v_2),(v_1,v_3),(v_1,v_4),(v_5,v_7),(v_6,v_7)\},\\
\eta_\R^b(2) = \{&(v_2,v_5),(v_2,v_7),(v_3,v_5),(v_3,v_6),(v_3,v_7),(v_4,v_6),(v_4,v_7)\}.
\end{align*}
\normalsize

Using scheduling $\eta_\R^a$, all paths are characterized by the same delay, namely $G_\R^a(z) = \frac{1}{z}$. Since $\Pi^a = 3$, the total delay introduced by the network is given by $\Pi^a \Delta = 30\ ms$. Using scheduling $\eta_\R^b$, paths are characterized by different delays, namely $G_\R^b(z) = \frac{0.6z+0.4}{z^2}$. Since $\Pi^b = 2$, the total delay introduced by the network is given by $\Pi^b \Delta = 20\ ms$. It is clear that there is a non-trivial tradeoff in the design of the scheduling function to obtain the minimum total network delay while obtaining simple transfer functions for the controllability graph, which motivates our modeling framework.
\end{example}


\section{Data rate model of communication links}\label{secMCNDataRateDef}

Consider a control scheme where actuation (resp. sensing) data generated by the controller (reps. by the sensor node that measures the output of the plant) is quantized and transmitted through the controllability (resp. observability) network. We assume that such quantization process is characterized by a quantization width $\delta_u$ (resp $\delta_y$) over a range of values bounded by $[-U_{max}, U_{max}]$ (resp. $[-Y_{max},Y_{max}]$). This is due to an implementation constraint, since the number of bits that we can use to represent the data is finite.

\begin{remark} In this paper, we only consider the effect of quantization as limited bandwidth constraints. It is out of the scope of this paper to characterize the effect of quantization on the dynamical properties of the closed loop system. For such results, the reader is referred to the existing literature on quantized control of linear systems, e.g. \cite{LiberzonTAC2007} and references therein.
\end{remark}

We assume that quantized data can be transmitted in the communication networks according to one of the following computational models implemented by the nodes. In the first model, each node computes the sum of all data received by the incoming links, and then transmits these data to the outgoing links multiplied by the corresponding link weight. In the second model, each node computes the weighted sum of all data received by the incoming links, and then transmits these data to the outgoing links. The advantage of the first model is that it allows each node to transmit with different data rates on the outgoing links, thus exploiting all available bandwidth on all links, while the second model restricts nodes to transmit at the same data rate on all outgoing links. On the other hand, as we will see in Section \ref{secCoDesignProblem}, the second model induces weaker conditions for solving efficiently the optimization problem of the closed loop system performance. It is worth to remark that the above computational models, even if they affect in a different way the data rate modeling, generate the same controllability and observability networks dynamics as in \eqref{eqGammaDef}.

\textbf{Computational model 1:} We first provide an example to show that the number of bits needed by a generic link $(v,v')$ to transmit data depends on the weight function $W_\R$.


\begin{example}\label{exDataRate}
Let us consider the controllability graph as in Figure \ref{figExamples}(b) and suppose that $\delta_u = 1$ and $U_{max} = 2$. Suppose that $W(v_1,v_3) = 1$: the link $(v_1,v_3)$ needs just 2 bits for representing the data. Suppose that $W(v_1,v_2) = 10^{-3}$: to avoid loss of precision when storing data in node $v_2$ we need to transmit on $(v_1,v_2)$ more bits with respect to $(v_1,v_3)$.
\end{example}

Given a controllability network $G_\R = ( V_{\R},E_{\R}, W_{\R} )$, we propose a model for relating the transmission data rate of each link with the weight function $W_{\R}$ and with the quantization parameters $\delta_u$, $U_{max}$, used to quantize the data at the input of the controllability network.
For each link $e=(v,v') \in E_\R$ we define a variable $\alpha_{e}$ as the sum of weights of all paths starting form $v_c$ and terminating with the link $e$:
\begin{align}\label{eqAlphaDef}
\alpha_{(v,v')}&=\sum_{d \in \DR}\alpha_{(v,v')}(d),\notag\\
\alpha_{(v,v')}(d)&=\sum_{
\scriptsize
\begin{array}[c]{c}\rho\in\chi_{\R}(d) \\ \rho=v_u,\ldots,v,v'\end{array}}
\normalsize
W_\R(\rho),
\end{align}
The data rate $R_e$ associated to data transmission on link $e$ is given by the number of bits needed to represent received data divided by the transmission time $\Delta$:
\begin{equation}\label{eqRvDef}
R_e=\Delta^{-1}\left\lceil log_2\left(\frac{2U_{max}}{\delta_u}\cdot\frac{|\alpha_e|}{\min\limits_{e \in E_\R} |\alpha_e|}\right)\right \rceil,
\end{equation}
where $\lceil \cdot \rceil$ is the ceiling function. Therefore, the data rate depends on the weight function $W_\R$ via equation \eqref{eqAlphaDef}. Note that also the variables $\boldsymbol\gamma_{\R}$ as defined in \eqref{eqGammaDef} are related to the variables $\alpha_e$ as defined in \eqref{eqAlphaDef}:
\begin{equation}\label{eqRelationGammaAlpha}
\gamma_{\R}(d)=\sum_{e\in inc(v_u)} \alpha_e(d),
\end{equation}
where $inc(v)=\{(\bar v, v) \in E_\R : \bar v \in V_\R\}$ is the set of incoming links in $v$. It is easy to check that the minimum achievable data rate is obtained for $\alpha_e = 1$, $\forall e \in E_\R$, which corresponds to the weight function defined by $W(v,v') = \frac{1}{|inc(v)|}$, $\forall (v,v') \in E_\R$. This choice is not necessarily the one that optimizes the performance of the closed loop system, as illustrated in Section \ref{secCoDesignProblem}.

\textbf{Computational model 2:} Using the same reasoning as above, we define for each $v \in V_\R \setminus v_u$ a variable $\alpha_v$ as follows:
\begin{align}\label{eqAlphaDef2}
\alpha_v=\sum_{d \in \DR}\alpha_v(d), \quad \alpha_v(d)=\sum_{
\scriptsize
\begin{array}[c]{c}\rho\in\chi_{\R}(d) \\ \rho=v_u,\ldots,v\end{array}
\normalsize
} W_\R(\rho),
\end{align}
with $\alpha_{v_c}=1$ since the controller node directly handles the quantized data at the input of the controllability network. Furthermore, since the actuator node is directly connected to the actuator and therefore does not transmit data in the network, we do not define $\alpha_{v_u}$. The transmission data rate of node $v$ and the variables $\gamma_{\R}(d), d \in \DR$ are given as follows:
\begin{align}
&R_v=\Delta^{-1}\left\lceil log_2\left(\frac{2U_{max}}{\delta_u}\cdot\frac{|\alpha_v|}{\min\limits_{\overline{v}\in V_{\R} \setminus v_u}{|\alpha_{\overline{v}}|}}\right)\right \rceil, \label{eqRvDef2}\\
&\gamma_{\R}(d)=\sum_{v \in pre(v_u)}W(v,v_u) \alpha_v(d), \label{eqRelationGammaAlpha2}
\end{align}
where $pre(v)=\{\overline{v}\,:\,(\overline{v},v)\in E_\R\}$ is the set of predecessor nodes of $v$. It is easy to check that the minimum achievable data rate is obtained for $\alpha_v = 1$, $\forall v \in V_\R \setminus v_u$, which corresponds to the weight function defined by $\sum\limits_{v' \in inc(v)}W(v',v)= 1$, $\forall v \in V_\R \setminus v_u$.

Depending on the use of the first or second computational model, we define the vector $\boldsymbol\alpha_\R$ respectively by $[\alpha_e, e \in E_\R] \in \mathbb R^{|E_\R|}$ and $[\alpha_v, v \in V_\R \setminus v_u] \in \mathbb R^{|V_\R|-1}$. The same reasoning can be used to characterize the data rates of the observability network.


\section{Co-design optimization problem} \label{secCoDesignProblem}

We consider in this section the problem of co-designing a digital controller $C(z)$ and the weight functions $W_\R$ and $W_\O$, in order to minimize the quadratic $\mathcal L_2$ norm of the error signal $e(k) = y(k) - r(k)$ while requiring the closed loop system to follow a step reference with zero steady state error in finite time (deadbeat control), with constraints on the maximum overshoot in the input and output signal of the plant and on the maximum data rate on the communication links.

\subsection{Controller and weight functions co-design for a given scheduling function} \label{subsecCoDesignGivenScheduling}

The following example motivates our investigation by illustrating that, for a given scheduling function, minimization of the quadratic $\mathcal L_2$ norm of the error signal, of the overshoot and of the data rate on the links are conflicting specifications.
\begin{example}\label{exOvershoot}
Let us consider a plant and a controllability graph as in Figure \ref{figExamples}(a), described respectively by the transfer functions $P(z)=\frac1{z+3}$ and $G_\R(z)=\frac{\gamma_\R(1)z+\gamma_\R(2)}{z^2}$. In this example we assume that communication nodes use the computational model 2 and, for simplicity, we will not consider the observability network. Consider a controller of the following form:
$$
C(z)=\frac{z^2(d_1z+d_0)}{(z-1)(z^2+c_1z+c_0)}.
$$
To design a deadbeat controller we impose $N_F+D_F=z^4$, where $F(z)=C(z)G_\R(z)P(z)$, thus obtaining the following overshoot of the plant output:
\begin{align*}
O_y&=
\max \left\{\begin{array}{c}
         \left|\frac{20\gamma_\R(1)\gamma_\R(2)+21\gamma_\R(1)^2}{3\gamma_\R(1)^2-\gamma_\R(2)^2+2\gamma_\R(1)\gamma_\R(2)}\right|, \\
         \left|\frac{3\gamma_\R(1)^2+20\gamma_\R(2)^2+20\gamma_\R(1)\gamma_\R(2)}{3\gamma_\R(1)^2-\gamma_\R(2)^2+2\gamma_\R(1)\gamma_\R(2)}\right|
       \end{array}\right\}
\end{align*}
Assume that $\delta_u=0.1$, $U_{max} = 500$, $\Delta=10ms$ and $\gamma_\R(1)=1$, then the overshoot is given by:
$$
O_y=\frac{20\gamma_\R(2)^2+20\gamma_\R(2)+3}{-\gamma_\R(2)^2+2\gamma_\R(2)+3}.
$$
When a cancelation occurs with an unstable pole of the plant (i.e. $\frac{\gamma_\R(2)}{\gamma_\R(1)}=3$), then the overshoot tends to infinity. We discuss now two co-design cases:

\noindent\textbf{Case 1:} Suppose to design $W_\R$, so that $\gamma_\R(1)=450$ and $\gamma_\R(2)=0.003$. Designing the controller variables $c_0,\ c_1,\ d_0,\ d_1$ to obtain deadbeat control, we obtain:
$$
\parallel e(k) \parallel_{\mathcal L_2}=6.245, \quad O_y=7, \quad R_{v_2} = 3.1kHz.
$$
\noindent\textbf{Case 2:} Suppose to design $W_\R$, so that $\gamma_\R(1)=1$ and $\gamma_\R(2)=1$. Designing the controller variables $c_0,\ c_1,\ d_0,\ d_1$ to obtain deadbeat control, we obtain:
$$
\parallel e(k) \parallel_{\mathcal L_2}=13.5508, \quad \quad O_y = 502.5, \quad R_{v_2} = 1.4kHz.
$$
The co-design of case 2, with respect to case 1, improves the transmission data rate of node $v_2$ at the cost of strongly increasing the overshoot and the $\mathcal L_2$ norm.
\end{example}

The above example shows that the co-design of controller and weight functions is non-trivial, and therefore motivates addressing the following optimization problems, for the computational models described in Section \ref{secMCNDataRateDef}.
%
%

\begin{problem}\label{probMCNOptimizationProblem}
Given a MCN $\mathcal{N}$, let $\P$, $(V_{\R},E_{\R})$, $\eta_{\R}$, $(V_{\O},E_{\O})$, $\eta_{\O}$ and $\Delta$ be given. Design a controller $C(z)$ and controllability and observability weight functions $W_\R$ and $W_\O$ to minimize the quadratic $\mathcal L_2$ norm of the error signal of the closed loop system, subject to the following constraints:
\begin{enumerate}
\item the closed loop system is asymptotically stable, with zero steady state error in finite time with respect to a step reference (deadbeat control);

\item the overshoots of the signals $u(k)$ and $y(k)$ are norm bounded;

\item the transmission data rate of each link (resp. of each node) is upper bounded.
\end{enumerate}
\end{problem}

Since we consider a step reference input, we design a digital controller $C(z)$ in the following form:
\begin{align}
C(z)&=M(z) z^{\DDR} z^{\DDO} C'(z),\notag\\
C'(z)&=\frac{d_sz^s+d_{s-1}z^{s-1}+\ldots+d_0}{(z-1) (z^m+c_{m-1}z^{m-1}+\ldots+c_0)}
\end{align}
where $M(z)$, $z^{\DDR}$ and $z^{\DDR}$ are used to delete the stable poles of the plant and of the controllability and observability networks. In order to obtain the smallest response time of the closed loop system, we design $C(z)$ with same degree in the numerator and denominator polynomials, therefore $m + 1 = s + deg(M(z)) + \DDR + \DDO$. Note that the closed loop system is forced to follow a step reference with zero steady state error in finite time if and only if
\begin{equation}\label{eqDeadbeatControl}
D_C D_{P'} + N_{C'} N_{G_\R} N_P N_{G_\O}=z^{l},
\end{equation}
where $l = m + r + 1 - \DDO$ is the response time. Moreover, we consider the weight functions $W_\R$ and $W_\O$ as additional variables of our optimization problem. In other words, we address co-design of the controller $C(z)$ and of the dynamics $G_\R(z),G_\O(z)$ due to the data flow through the networks. The variables of our optimization problem are defined by the following real vectors:
\begin{equation*}
\begin{array}{ll}
  \mathbf{c}=[c_{m-1}, c_{m-2},\cdots, c_0]; & \mathbf{d}=[d_{s}, d_{s-1},\cdots, d_0]; \\
  \mathbf{w_\R}=[W_{\R}(e), \forall e \in E_\R]; & \mathbf{w_\O}=[W_{\O}(e), \forall e \in E_\O].
\end{array}
\end{equation*}

\textbf{Computational model 1:} Problem \ref{probMCNOptimizationProblem} can be formalized as follows:
$$
\min\limits_{\bold{c}, \bold{d}, \bold{w_\R}, \bold{w_\O}}{||e(k)||^2_{\mathcal L_2}}
$$
subject to:
\begin{enumerate}
\item $D_C D_{P'} + N_{C'} N_{G_\R} N_P N_{G_\O}=z^{l}$;\\

\item $O_{u} \leq \overline{O}_{u}$, $O_{y} \leq \overline{O}_{y}$;\\

\item $R_e \leq \bar{R}_e$, $\forall e \in E_\R \cup E_\O$.
\end{enumerate}
The following propositions characterize the algebraic relation between the objective function and constraints in Problem \ref{probMCNOptimizationProblem}, and the optimization variables $\mathbf{c}$, $\mathbf{d}$, $\mathbf{w_\R}$, and $\mathbf{w_\O}$.
\begin{proposition}\label{propHardL2Norm}
The quadratic $\mathcal{L}_2$ norm of the error signal of the closed loop system is a polynomial function of $\bold{d}$ and $\bold{w_\R}$.
\end{proposition}
\begin{proof}
The Z transform of $y(k)$ in the closed loop system can be written as
\begin{align}\label{eqZPlantOutput}
Y(z)&=\frac{C(z) G_\R(z) P(z)}{1+C(z) G_\R(z) P(z) G_\O(z)} R(z)\notag\\
&=\frac{z^{\DDO} N_{C'} N_{G_\R} N_P}{D_C D_{P'} + N_{C'} N_{G_\R} N_P N_{G_\O}}R(z)\notag\\
&=\frac{q_{l-\nu}z^{l-\nu}+\ldots+q_1z+q_0}{z^{l}} R(z)
\end{align}
where $\nu = \min\{\DR\} + 1 \geq 2$ is the difference between the number of poles and the number of zeros of the closed loop transfer function, and $R(z)=A\frac z{z-1}$, $A\in\mathbb{R}$ is a step reference input. It follows that
\begin{equation*}
Y(z)=\left(\frac{q_{l-\nu}}{z^\nu} + \ldots+\frac{q_1}{z^{l-1}}+\frac{q_0}{z^l}\right)\frac{Az}{z-1}
\end{equation*}
and
\begin{align}\label{eqPlantOutput}
y(k) = Aq_{l-\nu}\delta_{-1}(k-\nu)+\ldots+Aq_0\delta_{-1}(k-l),
\end{align}
where $\delta_{-1}(k)$ is the unitary step. Therefore
\begin{align*}
e(k)= Aq_{l-\nu}\delta_{-1}(k-\nu)+ \ldots + Aq_0\delta_{-1}(k-l) - A\delta_{-1}(k).
\end{align*}
Thus,
\begin{align}\label{eqL2normFunction}
\big|\big| e(k)\big|\big|^2_{\mathcal L_2} =& \Big|\Big| \big(\underbrace{-A, \ldots, -A}_{\nu \text{ times}}, Aq_{l-\nu} - A, Aq_{l-\nu} +\notag\\+& Aq_{l-\nu-1} - A, \ldots, \sum\limits_{h=0}^{l-\nu} \big(Aq_h - A\big) \Big|\Big|^2_{\mathcal L_2}.
\end{align}
As shown in \eqref{eqZPlantOutput}, the parameters $q_h$ are sums of products of the coefficients of $N_{C'}$, $N_{G_\R}$, $N_P$, $D_{G_\O}$ as follows:
\begin{equation}\label{eqParametersOutput}
q_h=\sum_{i + j + (\DDR-\kappa) = h} b_i \cdot d_j \cdot \gamma_\R(\kappa),
\end{equation}
where $\gamma_\R(\kappa), \kappa \in \DR$ are polynomial functions of $\bold{w_\R}$ as in \eqref{eqGammaDef}. This concludes the proof.

\end{proof}

\begin{remark}
Note that, as shown in \eqref{eqL2normFunction}, the quadratic $\mathcal{L}_2$ norm of the error signal increases as $\nu = \min\{\DR\} + 1$ increases. This implies that the \emph{minimum} delay $\min\{\DR\}$ is the one that mainly affects the responsiveness of the system to variations of the reference signal, and is thus more significant than the other delays introduced by the network.
\end{remark}

\begin{proposition}\label{propHardDeadbeat}
Constraint 1 of Problem \ref{probMCNOptimizationProblem} is a polynomial constraint with respect to $\bold{c}$, $\bold{d}$, $\bold{w_\R}$ and $\bold{w_\O}$.
\end{proposition}
\begin{proof}
Equation \eqref{eqDeadbeatControl} can be rewritten as follows:
$$
z^{l} + (p_{l-1} -1) z^{l-1} + \ldots + p_{0} = z^{l},
$$
where the parameters $p_h$ are given by
\begin{align*}
p_h=\sum_{i + j + (\DDR - \kappa) + (\DDO - \iota) = h} b_i d_j \gamma_\R(\kappa) \gamma_\O(\iota)+\sum_{i+j=h}a_i + c_j,
\end{align*}
and where $\gamma_\R(\kappa), \kappa \in \DR$ and $\gamma_\O(\iota), \iota \in \DO$ are polynomial functions of $\bold{w_\R}$ and $\bold{w_\O}$ as in \eqref{eqGammaDef}. This completes the proof.
\end{proof}


\begin{proposition}\label{propHardOutputOvershoot}
Constraint 2 of Problem \ref{probMCNOptimizationProblem} is a polynomial constraint with respect to $\bold{d}$ and $\bold{w_\R}$.
\end{proposition}
\begin{proof}
By \eqref{eqPlantOutput} the overshoot $O_y$ can be written as follows:
\begin{align}\label{eqOvershootOutput}
O_y =& \Big|\Big| \Big( \underbrace{-A, \ldots, -A}_{\nu \text{ times}}, Aq_{l-\nu} - A, Aq_{l-\nu} + Aq_{l-\nu-1} - A,\notag\\
&\ldots, \sum\limits_{h=1}^{l-\nu} \Big(Aq_h - A\ \Big) \Big|\Big|_{\infty}.
\end{align}
Since parameters $q_h$ are as in \eqref{eqParametersOutput}, then the constraint $O_y \leq \overline{O}_y$ can be rewritten as $l-\nu+1$ constraints $|-A| \leq O_y, |Aq_{l-\nu} - A| \leq O_y,\ldots, \left|\sum\limits_{h=1}^{l-\nu} (Aq_h - A\ )\right| \leq O_y$, which are polynomial with respect to $\bold{d}$ and $\bold{w_\R}$.
Using the same reasoning, it is easy to prove that the same holds for the constraint on $O_{u}$.
\end{proof}

\begin{proposition}\label{propHardRate}
Constraint 3 of Problem \ref{probMCNOptimizationProblem} is a polynomial constraint with respect to $\bold{w_\R}$ and $\bold{w_\O}$.
\end{proposition}
\begin{proof}
The constraint $R_e \leq \bar R_e$, $e \in E_\R$ can be written as:
\begin{align}\label{eqAlphaConstraints}
&R_e=\Delta^{-1}\left\lceil log_2\left(\frac{2U_{max}}{\delta_u}\cdot\frac{|\alpha_e|}{\min\limits_{\overline{e}\in E_{\R}}{|\alpha_{\overline{e}}|}}\right)\right \rceil \leq \bar R_e \notag\\
&\Leftrightarrow\, \frac{(2U_{max})|\alpha_e|}{\delta_u \min\limits_{\bar{e} \in E_\R}{|\alpha_{\bar{e}}|}} \leq  2^{\bar R_e \Delta} \Leftrightarrow\, |\alpha_e| \leq \frac{\delta_u 2^{\bar R_e \Delta}}{2U_{max}}\min_{\bar{e} \in E_\R}|\alpha_{\bar{e}}| \notag\\
&\Leftrightarrow\, |\alpha_e|-\kappa_e\min_{\bar{e} \in E_\R}|\alpha_{\bar{e}}| \leq 0 \Leftrightarrow\, |\alpha_e|-\kappa_e |\alpha_{\bar{e}}| \leq 0,\ \forall \bar{e} \in E_\R,
\end{align}
where $\kappa_e \geq 0$, and by \eqref{eqAlphaDef} the variables $\alpha_e, e \in E_\R$ are polynomial functions of $\bold{w_\R}$. The proof for $e \in E_\O$ is similar.
\end{proof}


Propositions \ref{propHardL2Norm}-\ref{propHardRate} show that Problem \ref{probMCNOptimizationProblem} is a polynomial optimization problem with respect to the variables $\bold{c}$, $\bold{d}$, $\bold{w_\R}$ and $\bold{w_\O}$, which is well known to be NP-hard (see \cite{ParriloCDC2009} and references therein). Therefore, unless $P=NP$, there exist no polynomial-time algorithm to solve any instance of Problem \ref{probMCNOptimizationProblem}. In the following we provide conditions such that Problem \ref{probMCNOptimizationProblem} is a convex optimization problem.
\begin{lemma}\label{lemAlphaGammaWdataRateOnLinks}
Using the computational model 1, the following statements hold:
\begin{description}
\item[(i)] For any given vector $\boldsymbol\alpha_\R$, there always exists a unique vector $\bold{w_\R}$ that satisfies equation \eqref{eqAlphaDef};

\item[(ii)] Each component of $\boldsymbol\gamma_\R$ is a linear combination of the components of $\boldsymbol\alpha_\R$ if and only if
\begin{align}\label{eqDelayDisjoint}
&\forall e_1,e_2 \in inc(v), v \in V_\R \setminus v_u, \exists ! d \in \DR :\notag\\
&\forall \rho \in \bigcup\limits_{\bar d \in \DR \setminus d} \chi_\R(\bar d), e_1 \notin \rho, e_2 \notin \rho,
\end{align}
namely if and only if links belonging to paths characterized by different delays never merge in the same node.
\end{description}
\end{lemma}
\begin{proof}
\emph{(i):} By equation \eqref{eqAlphaDef} it follows that:
\begin{align}\label{eqFromWtoAlpha}
W_\R(v,v') \cdot {\sum\limits_{\bar v \in inc(v)} \alpha_{(\bar v,v)}} = {\alpha_{(v,v')}}.
\end{align}
Since the graph is acyclic, for any given $\alpha_{(v,v')}$ we can assign the variable $W_\R(v,v')$ so that equation \eqref{eqAlphaDef} is satisfied, without affecting the value of the variables $\alpha_{(\bar v,v)}$, $\bar v \in inc(v)$, which can thus be considered as constants in equation \eqref{eqFromWtoAlpha}. By applying this reasoning for all $(v,v') \in E_\R$ the result follows.

\emph{Sufficiency of (i):} If \eqref{eqDelayDisjoint} holds then each link $(v,v_u) \in E_\R$ will belong to paths characterized by the same delay $d_v$, namely $\alpha_{(v,v_u)} = \alpha_{(v,v_u)}(d_v)$. This implies that $\gamma(d) = \sum\limits_{v : d_v = d} \alpha_{(v,v_u)}$ for all $d \in \DR$.

\emph{Necessity of (ii):} If condition \eqref{eqDelayDisjoint} is false, then there exist at least two links $(v_1,v')$ and $(v_2,v')$ that respectively belong to paths $\rho_1 \in \chi_\R(d_1)$ and $\rho_2 \in \chi_\R(d_2)$, $d_1, d_2 \in \DR$. Then there exists at least one link $(v,v_u) \in E_\R$ that belongs both to $\rho_1$ and $\rho_2$. Let $W_\R(v',v) = \sum\limits_{\rho = v',\ldots,v}W_\R(\rho)$, then:
\begin{align*}
\alpha_{(v,v_u)} &= (\alpha_{v_1}(d_1) + \alpha_{v_2}(d_2)) W_\R(v',v) W_\R(v,v_u),\\
\gamma_\R(d_1) &= \alpha_{v_1}(d_1) W_\R(v',v) W_\R(v,v_u).
\end{align*}
Since $\alpha_{v_1} = \alpha_{v_1}(d_1)$ and $\alpha_{v_2} = \alpha_{v_1}(d_2)$, then $\gamma_\R(d_1) = \frac{\alpha_{v_1} \alpha_{(v,v_u)}}{\alpha_{v_1} + \alpha_{v_2}}$. This concludes the proof.
\end{proof}
The above proposition similarly holds for $\boldsymbol\alpha_\O$, $\boldsymbol\gamma_\O$ and $\bold{w_\O}$. A naive way to efficiently solve Problem \ref{probMCNOptimizationProblem} is by removing the dynamics on the networks:
\begin{proposition}
Let $|\DR| = 1$ and $|\DO| = 1$, then Problem \ref{probMCNOptimizationProblem} is a convex optimization problem with respect to the variables $\bold{c}$ and $\bold{d}$.
\end{proposition}
\begin{proof}
If $|\DR| = 1$ and $|\DO| = 1$, by Propositions \ref{propHardL2Norm}-\ref{propHardRate} it follows that the absence of $\boldsymbol\gamma_\R$ and $\boldsymbol\gamma_\O$ decouples the optimization variables $\bold{c}$, $\bold{d}$ and the variables $\bold{w_\R}$ and $\bold{w_\O}$ of the data rate constraints. By statement (i) of Proposition \eqref{eqDelayDisjoint}, it follows that the data rates can be arbitrarily assigned by designing appropriately $\bold{w_\R}$ and $\bold{w_\O}$. Finally, the objective function and the constraints 1 and 2 of Problem \ref{probMCNOptimizationProblem} become convex with respect to $\bold{c}$ and $\bold{d}$.
\end{proof}
However, as we illustrated in Example \ref{DelayExample}, the case $|\DR| = 1$ and $|\DO| = 1$ is not easily achievable. For this reason, we consider the following relaxation of Problem \ref{probMCNOptimizationProblem}:
\begin{problem}\label{probMCNOptimizationRelaxedProblem}
Let $(\bold{d}^*, \boldsymbol\alpha^*_\R)$ be the Pareto optimal value of the following optimization problem:
$$
\min\limits_{\bold{d}, \boldsymbol{\alpha_\R}}{||e(k)||^2_{\mathcal L_2}}
$$
subject to:
\begin{description}
\item[2)] $O_u \leq \overline{O}_u$, $O_y \leq \overline{O}_y$;\\

\item[3a)] $R_e \leq \bar{R}_e$, $\forall e \in E_\R$.

\end{description}
Given $\bold{d}=\bold{d}^*$ and $\boldsymbol{\alpha}_\R = \boldsymbol{\alpha}_\R^*$, design $\bold{c}$ and $\boldsymbol\alpha_\O$ such that the following hold:
\begin{description}
\item[1)] $D_C D_{P'} + N_{C'} N_{G_\R} N_P N_{G_\O}=z^{l}$;
\item[3b)] $R_e \leq \bar{R}_e, \forall e \in E_\O$.
\end{description}
\end{problem}
If Problem \ref{probMCNOptimizationRelaxedProblem} is feasible, then Problem \ref{probMCNOptimizationProblem} is feasible as well, and the Pareto optimal values of Problem \ref{probMCNOptimizationProblem} and of the first part of Problem \ref{probMCNOptimizationRelaxedProblem} are equal. This is due to the fact that the variables $\bold{c}$ and $\boldsymbol\gamma_\O$ do not appear in the objective function. Using the first computational model, we cannot state a priori conditions that guarantee feasibility of the second part of Problem \ref{probMCNOptimizationRelaxedProblem}: we will show that this is possible using the second computational model. If instead Problem \ref{probMCNOptimizationRelaxedProblem} is not feasible, then Problem \ref{probMCNOptimizationProblem} might be feasible, and the Pareto optimal value of the first part of Problem \ref{probMCNOptimizationRelaxedProblem} is a lower bound for the Pareto optimal value of Problem \ref{probMCNOptimizationProblem} (if they exist). This is due to the fact that, splitting the Problem \ref{probMCNOptimizationProblem}, we are neglecting some feasible solutions that require the joint design of all optimization variables.
\begin{proposition}\label{propWpositive}
Let $\bold{w_\R} > 0$, $\bold{w_\O} >0$, let \eqref{eqDelayDisjoint} hold both for $G_\R$ and $G_\O$, and let $|\bold{d}|=1$ (namely the numerator of the controller is a polynomial of zero degree). Using the computational model 1, the first part of Problem \ref{probMCNOptimizationRelaxedProblem} is a convex optimization problem with respect to $\boldsymbol\alpha_\R$, and the second part is a system of linear equations with linear constraints, with respect to $\bold{c}$ and $\boldsymbol\alpha_\O$.
\end{proposition}
\begin{proof}
Since $|\bold{d}|=1$ and each component of $\boldsymbol\gamma_\R$ is a sum of components of $\boldsymbol\alpha_\R$ by Lemma \ref{lemAlphaGammaWdataRateOnLinks}, then the quadratic $\mathcal L_2$ norm of the error signal and constraint 2 are convex with respect to $\boldsymbol\alpha_\R$. Since $\bold{w_\R} > 0$, $\bold{w_\O} >0$, then constraint 3a is linear constraint with respect to $\boldsymbol\alpha_\R$. Finally, since $\bold{d}=\bold{d}^*$ and $\boldsymbol{\alpha}_\R = \boldsymbol{\alpha}_\R^*$, constraint 1 is a system of linear equations in the variables $\bold{c}$ and $\boldsymbol\alpha_\O$, and constraint 3 is linear with respect to $\boldsymbol\alpha_\O$.
\end{proof}
The above proposition is based on the assumption that $\bold{w_\R} > 0, \bold{w_\O} >0$. If on one hand this makes Problem \ref{probMCNOptimizationProblem} convex, on the other hand it implies that $\boldsymbol\gamma_\R > 0, \boldsymbol\gamma_\O > 0$, which reduces the set of designable dynamics for $G_\R(z)$ and $G_\O(z)$. Therefore, we might need to use more complex controllers and observability networks in order to make Problem \ref{probMCNOptimizationProblem} feasible. It is also assumed that links belonging to paths characterized by different delays never merge in the same node, which restricts the set of admitted network topologies and scheduling functions. We will show that, if the communication nodes use the computational model 2, we do not need to raise these two assumptions.

\textbf{Computational model 2:} Problems \ref{probMCNOptimizationProblem} and \ref{probMCNOptimizationRelaxedProblem} can be formalized as above, by replacing the date rate constraint with $R_v \leq \bar{R}_v, \forall v \in V_\R \cup V_\O \setminus \{v_u,v_c\}$.
\begin{lemma}\label{lemComp2}
Using the computational model 2, and for any given $\boldsymbol\gamma_\R$ and $\boldsymbol\alpha_\R$, it is always possible to design $\bold{w_\R}$ such that equations \eqref{eqAlphaDef2} and \eqref{eqRelationGammaAlpha2} are satisfied.
\end{lemma}
\begin{proof}
It is easy to derive from \eqref{eqAlphaDef} that, for all $v \in V_\R \setminus \{v_c,v_u\}$, $\alpha_v = \sum\limits_{\bar v \in pre(v)} W(\bar v, v) \alpha_{\bar v}$. Since the graph is acyclic, and reasoning as in Lemma \ref{lemAlphaGammaWdataRateOnLinks}, it follows that given $\boldsymbol\alpha_\R$ it is possible to assign $\alpha_v$, $\forall v \in V_\R \setminus \{v_c,v_u\}$ by designing $W(v', v)$ for just one $v' \in pre(v)$. All the remaining $W(\bar v, v)$, $\bar v \in pre(v) \setminus v'$ can be considered as free variables, that we will use to design $\boldsymbol \gamma_\R$. Let $pre(v_u) = \bold{F}$: note that the weights $W(v,v_u)$, $v \in \bold{F}$ have not been used to assign the vector $\boldsymbol\alpha_\R$, therefore they can also be considered as free variables to design $\boldsymbol \gamma_\R$.

We claim that it is always possible to assign the vector $\boldsymbol\gamma_\R$ using the free variables defined above.
If $|\bold{F}| \geq |\DR|$, then it is clearly possible to assign $\boldsymbol \gamma_\R$ by arbitrarily choosing any $|\DR|$ weights $W(v,v_u)$, $v \in \bold{F}$.

Assume instead that $|\bold{F}| < |\DR|$. Let $\DR^1$ be the set of delays $d \in \DR$ such that there exists at least one node $v \in \bold{F}$ with $\alpha_v = \alpha_v(d)$, i.e. one node that belongs to paths characterized by a unique delay. By definition of $\DR^1$ it follows that $|\bold{F}| \geq |\DR^1| + 1$, since all paths associated to delays in $\DR^1$ do not merge before reaching node $v_u$. Let $\bold{F_1} \subseteq \bold{F}$ be the set of nodes belonging to paths characterized by a unique delay: it follows that we can assign $\gamma_\R(d)$ for all $d \in \DR^1$ by designing the weight $W_\R(v,v_u)$, for an arbitrary $v \in \bold{F_1}$ such that $\alpha_v = \alpha_v(d)$.

Let $\DR^2 = \DR \setminus \DR^1$, and let $\bold{F_2} = \bold{F} \setminus \bold{F_1}$. Any node $v \in \bold{F_2}$ belongs to paths characterized by a finite number $\mu \leq |\DR^2|$ of different delays. However, as discussed above, this implies that there exist at least $\mu$ free variables associated to $W(v,v_u) \alpha_v$, i.e. $\mu - 1$ for all merges of different delays and one more for the variable $W_\R(v,v_u)$. By repeating this reasoning for all $v \in \bold{F_2}$, it follows that we have enough free variables to assign $\gamma_\R(d)$ for all $d \in \DR^2$.
\end{proof}
By the above lemma it follows that constraints 3a and 3b are decoupled from Problem \ref{probMCNOptimizationRelaxedProblem}, and is therefore always possible to design $\bold{w}_\R$ and $\bold{w}_\O$ to assign the optimal data rates both in the controllability and observability graphs as illustrated in Section \ref{secMCNDataRateDef}.
\begin{proposition}\label{propComp2ConvexityConditions}
Let $|\bold{d}|=1$. Using the computational model 2, the first part of Problem \ref{probMCNOptimizationRelaxedProblem} is a convex optimization problem with respect to $\boldsymbol\gamma_\R$, and the second part is a system of linear equations with respect to $\bold{c}$ and $\boldsymbol\gamma_\O$.
\end{proposition}
\begin{proof}
If $|\bold{d}|=1$, by Lemma \ref{lemComp2} it directly follows that the quadratic $\mathcal L_2$ norm of the error signal and constraint 2 are convex with respect to $\boldsymbol\gamma_\R$, and that constraint 1 is a system of linear equations in the variables $\bold{c}$ and $\boldsymbol\gamma_\O$. Moreover, by Lemma \ref{lemComp2} it follows that  constraints 3a and 3b are decoupled from the objective function and all constraints of Problem \ref{probMCNOptimizationRelaxedProblem}, and therefore they disappear.
\end{proof}
The following proposition provides a sufficient condition on topology, scheduling and routing of the observability graph, such that Problem \ref{probMCNOptimizationProblem} and Problem \ref{probMCNOptimizationRelaxedProblem} are equivalent.
\begin{proposition}
Let $|\bold{d}|=1$ and $|\DO| \geq r+1$. Using the computational model 2, Problem \ref{probMCNOptimizationProblem} and Problem \ref{probMCNOptimizationRelaxedProblem} are equivalent.
\end{proposition}
\begin{proof}
The condition $|\DO| \geq r+1$ guarantees the existence of a solution on the variables $\bold{c}$ and $\boldsymbol\gamma_\O$ for the system of linear equations defined by constraint 1. Since by Lemma \ref{lemComp2} the constraints 3, 3a and 3b disappear, and since the variables $\bold{c}$ and $\boldsymbol\gamma_\O$ only appear in constraint 1, then constraint 1 can be decoupled from Problem \ref{probMCNOptimizationProblem}, and can also be decoupled from the second part of Problem \ref{probMCNOptimizationRelaxedProblem}. As a consequence, Problem \ref{probMCNOptimizationProblem} and the first part of Problem \ref{probMCNOptimizationRelaxedProblem} become equivalent.
\end{proof}

\subsection{Scheduling design} \label{subsecDesignScheduling}

\begin{problem}\label{probSchedulingOptimizationProblem}
Given a MCN $\mathcal{N}$, let $\P$, $(V_{\R},E_{\R})$, $(V_{\O},E_{\O})$ and $\Delta$ be given. Design scheduling functions $\eta_{\R}$ and $\eta_{\O}$, a controller $C(z)$, and controllability and observability weight functions $W_\R$ and $W_\O$, to minimize the quadratic $\mathcal L_2$ norm of the error signal of the closed loop system, subject to the constraints of Problem \ref{probMCNOptimizationProblem}.
\end{problem}
The set $\Gamma$ of admissible scheduling functions is finite, and its cardinality is exponential with respect to the number of communication links and to the scheduling period. Therefore, Problem \ref{probSchedulingOptimizationProblem} can be solved by computing the Pareto optimal value $\mathcal L_2^*(\eta_\R, \eta_\O)$, $\forall (\eta_\R, \eta_\O) \in \Gamma$, using the results obtained in the above section. This induces a total ordering in the set $\Gamma$, and we can define the set of optimal scheduling functions
$$
\Gamma^* = \big\{(\eta_\R^*, \eta_\O^*) \in \Gamma \colon (\eta_\R^*, \eta_\O^*) = \underset{(\eta_\R, \eta_\O) \in \Gamma}{\operatorname{argmin}} \mathcal L_2^*(\eta_\R, \eta_\O)\big\}.
$$

\begin{example}\label{exFinalExample}
Let us consider a plant, a controllability graph $G_\R$, and parameters $\delta_u$, $\Delta$ and $U_{max}$ as in Example \ref{exOvershoot}. We also consider here an observability graph $G_\O$ equal to $G_\R$, and the following observability scheduling function

\scriptsize
\begin{align*}
\eta_\O(1) = \{&(v_1,v_2),(v_1,v_3),(v_1,v_4),(v_2,v_5),(v_2,v_7),\\
               &(v_3,v_5),(v_3,v_6),(v_3,v_7),(v_4,v_6),(v_4,v_7),(v_5,v_7),(v_6,v_7)\},
\end{align*}
\normalsize
that is associated to the following transfer function:
$$
G_\O(z)=\frac{\gamma_\O(2)z+\gamma_\O(3)}{z^3}.
$$
We consider the following three scheduling functions:
\scriptsize
\begin{align*}
\eta_\R^a(1) = \{&(v_1,v_2),(v_1,v_3),(v_1,v_4),(v_5,v_7),(v_6,v_7)\},\\
\eta_\R^a(2) = \{&(v_2,v_5),(v_2,v_7),(v_3,v_5),(v_3,v_6),(v_3,v_7),(v_4,v_6),(v_4,v_7)\},\\
\eta_\R^b(1) = \{&(v_1,v_2),(v_1,v_3),(v_1,v_4),(v_2,v_7),(v_3,v_7),(v_4,v_7),(v_5,v_7),\\
                 &(v_6,v_7)\},\\
\eta_\R^b(2) = \{&(v_2,v_5),(v_3,v_5),(v_3,v_6),(v_4,v_6),\},\\
\eta_\R^c(1) = \{&(v_1,v_2),(v_1,v_3),(v_1,v_4),(v_2,v_5),(v_3,v_5),(v_3,v_6),(v_4,v_6),\\
                 &(v_4,v_7),(v_6,v_7)\},\\
\eta_\R^c(2) = \{&(v_2,v_7),(v_3,v_7),(v_5,v_7)\},
\end{align*}
\normalsize
that are associated to the following transfer functions:
\begin{align*}
G_{\R^a}(z)&=\frac{\gamma_\R(1)z+\gamma_\R(2)}{z^2};\\
G_{\R^b}(z)&=\frac{\gamma_\R(2)}{z^2};\\
G_{\R^c}(z)&=\frac{\gamma_\R(1)z^2+\gamma_\R(2)z+\gamma_\R(3)}{z^3}.
\end{align*}
Let $\bar O_u = \bar O_y = 10$, and $1.5\,kHz\leq\bar R_e\leq3\,kHz,\ \forall\ e\in E_\R \cup E_\O$. We chose a controller with $|\mathbf{d}| = 1$ and $|\mathbf{c}| = 4$ for $\eta_\R^a$ and $\eta_\R^b$, and with $|\mathbf{d}|=1$ and $|\mathbf{c}| = 5$ for $\eta_\R^c$.
We assumed that the communication nodes use the computational model 1, and solved Problem \ref{probMCNOptimizationProblem} for each scheduling using \texttt{CVX}, a package for specifying and solving convex programs \cite{gb08,CVX}.

\begin{figure}[ht]
\begin{center}
\vspace{-0.3cm}
\includegraphics[width=0.5\textwidth]{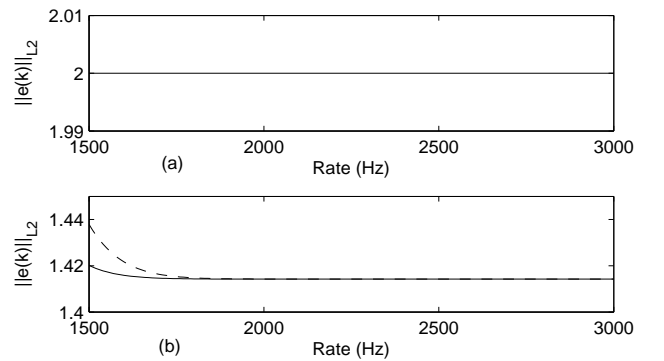}
\vspace{-0.7cm}
\caption{$\parallel e(k) \parallel_{\mathcal L_2}$ norm vs. maximum allowed data rate. (a) solid line: $\eta_\R^b$; (b) solid line: $\eta_\R^a$; (b) dashed line: $\eta_\R^c$.}\label{Rate_3casi}
\vspace{-0.5cm}
\end{center}
\end{figure}
Figure \ref{Rate_3casi} shows the $\mathcal L_2$ norm of the error as a function of the data rate bound $\bar R_e$, for the three scheduling defined above. Scheduling $\eta_\R^b$ is characterized by a single delay, and has the worst performance, probably because it introduces in the optimization problem only one degree of freedom. However, $\eta_\R^c$ performs worse than $\eta_\R^a$ despite the fact that it introduces more degrees of freedom, probably because the additional free variables are redundant to reach the minimum $\mathcal L_2$ norm and only increase the network delays leading to a worse performance. In conclusion, the scheduling that guarantees the best $\mathcal L_2$ norm while respecting the constraints is $\eta_\R^a$. Figure \ref{Rate_3casi} also shows that it is useless to increase the transmission data rate on the links to values greater than $1700 kHz$, since the performance does not improve after that point.
\end{example}


\section{Conclusions} \label{secConclusions}

Given a MCN, we addressed the problem of co-designing a digital controller and the network parameters (topology, scheduling and routing) to guarantee stability and maximize a performance metric on the transient response to a step input, with constraints on the bandwidth. We showed that the above problem is a polynomial optimization, which is generally NP-hard, and we provided sufficient conditions such that it reduces to a convex optimization.


\bibliographystyle{plain}
\bibliography{mcnbib}


\end{document}